\documentclass[12pt]{amsart}
\usepackage{graphicx, amssymb}
\title{Counting real curves with passage/tangency conditions}
\author{Sergei Lanzat}
\address{Department of Mathematics, Technion- Israel
Institute of Technology, Haifa 32000, Israel}
\email{serjl@tx.technion.ac.il, polyak@math.technion.ac.il}
\author{Michael Polyak}

\makeatletter

\newcommand{\FAT}[1]{\mbox{{$\mathbb{#1}$}}}
\newcommand{\Fat}[1]{\mbox{{$\scriptstyle\mathbb{#1}$}}}
\newcommand{\CL}[1]{\mbox{{$\mathcal{#1}$}}}
\newcommand{\Cl}[1]{\mbox{{$\scriptstyle\mathcal{#1}$}}}

\newcommand{\ZZ}{\FAT{Z}}
\newcommand{\DD}{\FAT{D}}
\newcommand{\PP}{\FAT{P}}
\newcommand{\RR}{\FAT{R}}
\newcommand{\CC}{\FAT{C}}
\newcommand{\SSS}{\mathbb{S}}
\newcommand{\ST}{\mathbb{S}\mathrm{T}}
\newcommand{\St}{\scriptstyle\mathbb{S}\mathrm{T}}
\newcommand{\PT}{\mathbb{P}\mathrm{T}}
\newcommand{\rr}{\Fat{R}}

\newcommand{\cl}[1]{\overline{#1}}
\newcommand{\minus}{\smallsetminus}
\newcommand{\eps}{\varepsilon}
\DeclareMathOperator{\ind}{ind}

\renewcommand{\(}{\left(}
\renewcommand{\)}{\right)}
\newcommand{\be}{\begin{itemize}}
\newcommand{\ee}{\end{itemize}}

\newtheorem{thm}{Theorem}[section]
\newtheorem{thm*}{Theorem}
\newtheorem{lem}[thm]{Lemma}
\newtheorem{prop}[thm]{Proposition}
\newtheorem{cor}[thm]{Corollary}
\newtheorem{rem}[thm]{Remark}

\begin{document}

 \begin{abstract}
We study the following question: given a set $\CL{P}$ of $3d-2$
points and an immersed curve $\Gamma$ in the real plane $\RR^2$, all
in general position, how many real rational plane curves of degree
$d$ pass through these points and are tangent to this curve. We
count each such curve with a certain sign, and present an explicit
formula for their algebraic number. This number is preserved under
small regular homotopies of a pair $(\CL{P}, \Gamma)$ but jumps (in
a well-controlled way) when in the process of homotopy we pass a
certain singular discriminant. We discuss the relation of such
enumerative problems with finite type invariants. Our approach is
based on maps of configuration spaces and the intersection theory in
the spirit of classical algebraic topology.
 \end{abstract}

\thanks{Both authors were partially supported by the ISF grant 1343/10.}

\maketitle

\section{Introduction}
\subsection{History}
A classical problem in enumerative geometry is the study of the
number of certain algebraic curves of degree $d$ passing through
some number of points in the affine or projective plane. This
question is not very interesting if we consider all curves of degree
$d$, due to the fact that the set of such curves forms the
projective space of dimension $\frac12 d(d+3)$,  so the question of
passing through points is simply a question of solving a system of
linear equations. Thus, one usually asks this question about some
families of algebraic curves of degree $d$, e.g., curves of a fixed
genus $g$. In particular, there is an old question of determining
the number $N_d$ (resp. $N_d(\RR)$) of rational, i.e. genus $g=0$,
curves of degree $d$ passing through $3d-1$ points in general
position in $\CC\PP^2$ (resp. $\RR\PP^2$). Here $3d-1$ is complex
(resp. real) dimension of an algebraic variety of irreducible
rational curves of degree $d$ in $\CC\PP^2$ (resp. $\RR\PP^2$).

The numbers $N_1=N_2=N_1(\RR)=N_2(\RR)=1$ go back to antiquity;
$N_3=12$ was computed by J.~Steiner in $1848$. The late $19$-th
century was the golden era for enumerative geometry,  and
H.G.~Zeuthen in $1873$ could compute the number $N_4=620$. By then,
the art of resolving enumerative problems had attained a very high
degree of sophistication and, in fact, its foundations could no
longer support it. Hilbert asked for rigorous foundation of
enumerative geometry, including it as the $15$-th problem in his
list.

The $20$-th century witnessed great advances in intersection theory.
In the seventies and eighties, a lot of old enumerative problems
were solved and many classical results were verified. However, the
specific question of determining the numbers $N_d$ turned out to be
very difficult. In fact, in the eighties only one more of the
numbers was unveiled: the number of quintics $N_5=87304$.

The revolution took place around $1994$ when a connection between
theoretical physics (string theory) and enumerative geometry was
discovered. As a corollary,  M.~Kontsevich and Yu.~Manin in \cite{K-M}
(see also \cite{F-P}) gave a solution in terms of a recursive formula
$$N_d=\sum\limits_{d_1+d_2=d, \;d_1, d_2>0}N_{d_1}N_{d_2}\left(d_1^2d_2^2
\binom{3d-4}{3d_1-2}-d_1^3d_2\binom{3d-4}{3d_1-1}\right).$$

But all these advances were done in the complex algebraic geometry.
In the real case the situation is different. Until $2000$ nothing
was known about $N_d(\RR)$ for $d\geq 3$. In $2000$ A.~Degtyarev and
V.~Kharlamov \cite{D-K} showed that $N_3(\RR)$ may be $8, 10$ or
$12$, depending on the configuration of $8=3\cdot 3-1$ points in
$\RR\PP^2$. This result reflects a general problem of a real
enumerative geometry: such numbers are usually not constant, but do
depend on a configuration of geometrical objects. A natural way to
overcome this difficulty is to try to assign some signs and
multiplicities to objects in question so that the corresponding
algebraic numbers remain constant. Already in the work of
Degtyarev-Kharlamov one can see that one can assign certain
multiplicities (signs) to real cubics passing through a given $8$ points, 
so that the weighted sum of these cubics is independent on the
configuration of points. In $2003$ J.-Y.~Welschinger \cite{W1} found
a way to assign signs to real rational curves of any degree. 
Welschinger's sign $w_C$ of a real rational curve $C$ is defined as 
$w_C=(-1)^{m(C)}$, where $m(C)$ is the number of solitary points 
of $C$ (called the mass of $C$).
Welschinger's main theorem states that the corresponding weighted 
sum $W_d=\sum_C w_C$ of all curves passing through the given points 
is independent of the choice of points. The number $W_d$ is called the 
Welschinger's invariant. In particular, $|W_d|$ gives a lower bound for 
the actual number $N_d(\RR)$ of real rational plane curves passing 
through any given set of $3d-1$ generic points. In the case of cubics 
$(d=3)$ from the Degtyarev-Kharlamov theorem one can see that 
$W_3=8$.

The question of passing through some number of points is the
simplest one. The next step is to ask about the number of algebraic
curves passing through some number of points and tangent to some
given algebraic curves. In particular, in 1848, J.~Steiner \cite{S}
asked how many conics are tangent to five given conics in
$\CC\PP^2$. The correct answer of $3264$ was found by M.~ Chasles
\cite{C} in 1864. In 1984, W.~Fulton  asked how many of these conics
can be real and in 1997, F.~ Ronga, A.~ Tognoli and Th.~ Vust
\cite{R-T-V} proved that all $3264$ conics can be real. Another
celebrated problem is due to Zeuthen. Given $l$ lines and $k=d(d +
3)/2-1$ points in $\CC\PP^2$, the Zeuthen number $N_d(l)$ is the
number of nonsingular complex algebraic curves of degree $d$ passing
through the $k$ points and tangent to the $l$ lines. It does not
depend on the chosen generic configuration $C$ of points and lines.
If, however, points and lines are real, the corresponding number
$N^{\rr}_d (l,C)$ of real curves usually depends on the chosen
configuration. For $l=1$, it was shown by F.~ Ronga \cite{R} that
the real Zeuthen problem is maximal: there exists a configuration
$C$ such that $N^{\rr}_d (1,C) = N_d(1)$. For $l=2$ a similar
maximality result was obtained by B.~ Bertrand \cite{B} using
Mikhalkin's tropical correspondence theorem. In the complex case in
$1996$ L.~Caporaso and J.~Harris found the recursive formulas in the
spirit of M.~Kontsevich for such tangency questions. In the real
case that kind of questions of tangency is quite new and the serious
development is just beginning. J.-Y.~Welschinger in \cite{W3}
considered curves in $\RR\PP^2$ passing through a generic set of
points and tangent to a non-oriented smooth simple zero-homologous
curve. See Section \ref{sub:intro_main} for a detailed discussion of
Welschinger's results and a comparison with the present work.

We are interested to merge rigid algebro-geometric objects with
flexible objects from smooth topology. We count algebraic curves in
$\RR^2$ that pass through a generic set of points and are tangent to
an oriented immersed curve. In addition, we relate the dependence on
a chosen configuration to the theory of finite type invariants.

\subsection{Motivation}\label{sub:motivation}
Let us start with a toy model: consider the case $d=1$.

Let $L$ be a set of lines in $\RR^2$ passing through a fixed point
$p$ and tangent to a (generic) oriented immersed plane curve
$\Gamma$. The problem is to introduce a sign $\eps_l$ for each such
line $l\in L$ so, that the total algebraic number $N_1(p, \Gamma)=\sum_{l\in
L}\eps_l$ of lines does not change under homotopy of $\Gamma$ in
$\RR^2\minus p$. It is easy to guess such a sign rule. Indeed, under
a deformation shown in Figure \ref{fig:morse}a, two new lines
appear, so their contributions to $N_1(p, \Gamma)$ should cancel out.
Thus, their signs should be opposite and we get the sign rule shown in Figure
\ref{fig:morse}b. Note that only the orientation of $\Gamma$ is used
to define it; $l$ is not oriented.

\begin{figure}[htb]
    \centering
    \includegraphics[width=5in]{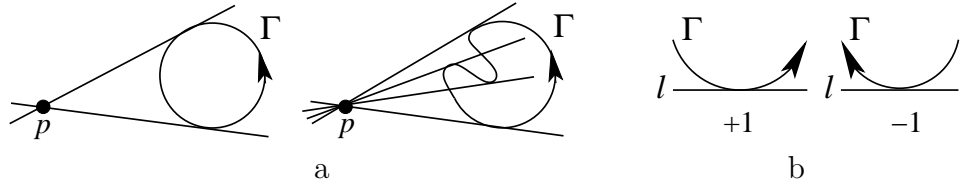}
\centerline{\hspace{0.8in}a\hspace{2.4in}b}
\caption{Counting lines with signs.\label{fig:morse}}
\end{figure}

Suppose that $p$ is close to infinity (i.e., lies in the unbounded
region of $\RR^2\minus\Gamma$). In this case we get
$N_1(p, \Gamma)=2\ind(\Gamma)$, where $\ind(\Gamma)$ is the Whitney
index (a.k.a. rotation number) of $\Gamma$, i.e. the number of turns made 
by the tangent vector as we pass once along $\Gamma$ following the
orientation. In other words, $\ind(\Gamma)$ equals to
the degree of the Gauss map $G_{\Gamma}:\SSS^1\to\SSS^1$ given by
$G_{\Gamma}(t)=\cfrac{\gamma^\prime(t)}{\|\gamma^\prime(t)\|}$, where
$\gamma:\SSS^1\looparrowright\RR^2$ is a parametrization of $\Gamma$.
Hence, the Whitney index can be calculated as an algebraic number of
preimages of a regular value $\xi\in\SSS^1$ of the Gauss map $G_{\Gamma}$,
see Figure \ref{fig:indices}a.

\begin{figure}[htb]
\centering
    \includegraphics[height=1.4in]{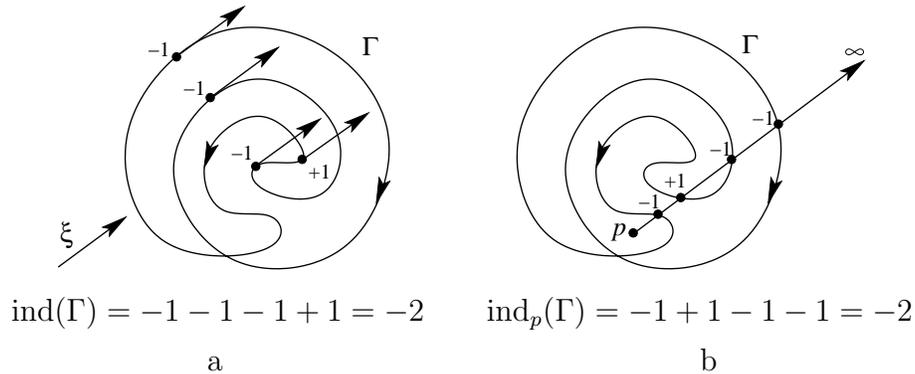}
    $$\ind(\Gamma)=-1-1-1+1=-2 \hspace{0.3in}\ind_p(\Gamma)=-1+1-1-1=-2$$
    $$\text{a}\hspace{2.5in}\text{b}$$
\caption{Whitney index of $\Gamma$ and index of $p$ w.r.t. $\Gamma$.
\label{fig:indices}}
\end{figure}

While $N_1(p, \Gamma)$ is preserved under homotopies of $\Gamma$ in $\RR^2\minus
p$, it changes when $\Gamma$ passes through $p$, see Figure \ref{fig:jump}.

\begin{figure}[htb]
    \centering
    \includegraphics[height=0.8in]{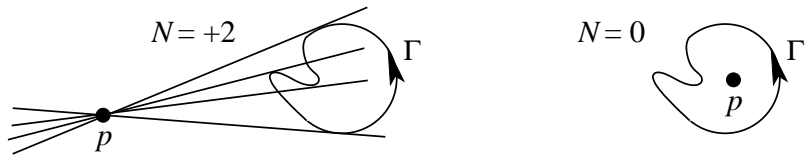}

\caption{Counting lines with signs.\label{fig:jump}}
\end{figure}

The compensating term is also easy to guess and we finally obtain
\begin{equation}\label{eq:N motivation}
N_1(p, \Gamma)=2\ind(\Gamma)-2\ind_p(\Gamma)\ .
\end{equation}
Here the index $\ind_p(\Gamma)$ of $p$ w.r.t. $\Gamma$ is the number
of turns made by the vector connecting $p$ to a point $x\in\Gamma$,
as $x$ passes once along $\Gamma$ following the orientation. It may
be computed as the intersection number $I([p,\infty],\Gamma;\RR^2)$
of a $1$-chain $[p,\infty]$ (i.e. an interval connecting $p$ with a point near
infinity of $\RR^2$) with an oriented $1$-cycle $\Gamma$ in $\RR^2$.
See Figure \ref{fig:indices}b.

The appearance of $\ind(\Gamma)$ and $\ind_p(\Gamma)$ in the above
formula comes as no surprise: in fact, these are the only invariants
of the curve $\Gamma$ under its homotopy in the class of immersions
in $\RR\minus p$. These are the simplest finite type invariants of
plane curves, see \cite{Arn}.

In this simple example we see two main distinctive differences of
real enumerative problems vs. complex problems of a similar
passage/tangency type. Firstly, in the real case we are to count
algebraic curves under the consideration with signs. Secondly, over
$\CC$ the answer is a number which does not depend on the relative
position of the set of points and the curve $\Gamma$. Over $\RR$,
however, the answer depends on the configuration: it is preserved
under small deformations of $\Gamma$ and the set of points, but
experiences certain (well-controlled) jumps when the configuration
crosses certain singular discriminant in the process of homotopy.
Thus, in the general case for similar enumerative problems we should
not expect to get an answer as one number, but rather as a
collection of numbers, depending on the relative configuration of
points and the smooth curve.

Two main questions in this kind of problems are

\be

\item[1.] How to find such sign rules, i.e. how to assign a
certain sign to each algebraic curve under consideration, so that
the total algebraic number is invariant under small deformations?

\item[2.] How does the singular discriminant looks like, and what is
the explicit structure of the formula for the algebraic number of
curves?

\ee

\subsection{Main results and the structure of the paper}
 \label{sub:intro_main}
We count the algebraic number of real plane rational nodal curves of
degree $d$ passing through a given set of $3d-2$ generic points and
tangent to a generic immersed curve in the plane $\RR^2$. We get a
number, which does not depend on a regular homotopy of the curve in
a complement of a certain singular discriminant, see Subsection
\ref{subsect:generalposition}. As the curve passes through the
discriminant, this number changes in a well-controlled way, so that
it defines a finite type invariant of degree one, see Section
\ref{sec:FTI}.
A mixture of rigid algebro-geometric objects with smooth topology
gives to our problem a curious flavor, leading to a nice merging of
features and techniques originating in both of these fields. In
particular, this type of passage/tangency problems turns out to be
intimately related to a theory of finite type invariants of plane
curves, similarly to the toy case of $d=1$ considered in Subsection
\ref{sub:motivation} above.

For this we count rational nodal curves with signs and add certain
correction terms, which come from degenerate cases of nodal,
reducible and cuspidal curves. We use Welschinger's signs and show
an easy way to produce new signs suitable for tangency questions.
The technique of proofs uses the concept of configuration spaces and
the intersection theory in the spirit of classical differential
topology.

The question of passage/tangency conditions for real rational plane
algebraic curves was considered earlier by J.-Y.~Welschinger in
\cite{W3}. He considered projective curves in $\RR\PP^2$ passing
through a generic set  $\CL{P}$ and tangent to a non-oriented smooth
simple zero-homologous curve $\Gamma$. In \cite[Remark $4.3(3)$]{W3}
the author suggested the generalization to the case of a
non-oriented smooth immersed  curve $\Gamma$, which bounds an
immersed disk; unfortunately, this formula does not extend to
arbitrary immersed curves, e.g. to a figure-eight curve. There is a
number of differences between \cite{W3} and the present work.
Firstly, we consider oriented curves in $\RR^2$ (thus adding
orientations both to the curve and to the ambient manifold).
Secondly, we consider immersed curves. Finally, in contrast with
\cite{W3}, where the author used 4-dimensional symplectic geometry
and hard-core techniques from the theory of moduli spaces of
pseudo-holomorphic curves, we use down to earth classical tools 
of differential topology. In this way we also get a clear geometric
interpretation of Welschinger's number $w_C$ as the orientation of a
certain surface in $\SSS \mathrm{T}^*\RR^2$ (i.e. the manifold of
oriented contact elements of the plane), which parameterizes real
rational algebraic curves passing through $\CL{P}$.

The paper is organized in the following way. In Section 2 we
introduce objects of our study, define signs of tangency, list the
requirements of a general position, and formulate the main theorem.
Section 3 is dedicated to the proofs. We interpret the desired
number of curves as a certain intersection number; the main claim
follows from different ways of its calculation. In Section 4 we
discuss a relation of real enumerative geometry to finite type
invariants.


\section{The Statement of the Main Result. Sketch of the Proof.}

\subsection{Curves and points in general position.
 \label{subsect:generalposition}}
Let $\CL{P}=\{p_i\}_{i=1}^{3d-2}$, $p_i\in\RR^2$,
be a $(3d-2)$-tuple of (distinct) points in
$\RR^2$. Consider the following sets $\CL{C}_{\Cl{P}}$,
$\CL{T}_{\Cl{P}}$, $\CL{R}_{\Cl{P}}$ in $\RR^2\minus\CL{P}$
determined by $\CL{P}$:
 \be
\item[$(i)$] We say that
$p\in\CL{C}_{\Cl{P}}$ (resp. $p\in\CL{T}_{\Cl{P}}$), if
there exists an irreducible rational curve of degree $d$, which
passes through $\CL{P}$, has a cusp (resp. tacnode or a
triple point) at $p$ and whose remaining singularities are nodes in
$\RR^2\minus\CL{P}$.
\item[$(ii)$] We say that $p\in\CL{R}_{\Cl{P}}$, if
there exists a reducible curve $C_1\cup C_2$ of degree $d$ with
$p\in C_1\cap C_2$, which passes through $\CL{P}$ and such that
$C_1$ and $C_2$ are irreducible rational nodal curves with nodes in
$\RR^2\minus\CL{P}$, which intersect transversely at their
non-singular points.
 \ee

For $p\in\CL{P}$ denote by $\CL{D}(p)$ the set of all irreducible 
rational nodal curves of degree $d$, which pass through $\CL{P}$, 
have a crossing node at $p$ and whose remaining nodes are in 
$\RR^2\minus\CL{P}$. Denote by $\CL{D}(\CL{P})$ the set 
$\CL{D}(p)$ together with curves listed in $(i)-(ii)$ above. Denote also
$\mathfrak{S}:=\CL{P}\cup\CL{C}_{\Cl{P}}\cup\CL{R}_{\Cl{P}}$.

Suppose that the $(3d-2)$-tuple $\CL{P}$ is in general position.
Explicitly, we will assume that the following conditions hold:

\be
\item[1.] For any $k<d$, no $3k$ points from $\CL{P}$ lie on one
rational curve of degree $k$.
\item[2.] The set $\mathfrak{S}$ is finite. Every point $p$ in
$\CL{C}_{\Cl{P}}\cup\CL{R}_{\Cl{P}}$ lies on
exactly one curve from $\CL{D}(\CL{P})$. For every $p\in\CL{P}$ 
the set $\CL{D}(p)$ is finite.
\item[3.] All rational curves of degree $d$ passing through
$\CL{P}$ are either irreducible nodal, with nodes in
$\RR^2\minus\mathfrak{S}$, or belong to $\CL{D}(\CL{P})$.
%
\ee
\noindent Now, let $\Gamma$ be a
generic immersed oriented curve in $\RR^2$ in general position
w.r.t. $\CL{P}$. Let us spell this requirement in more details. By a
general position we mean that

\be
\item[4.] The curve $\Gamma$ is  an immersed curve with a finite 
number of double points of transversal self-intersection as the only 
singularities.

\item[5.] The curve $\Gamma$ intersects each of the curves from the
set $\CL{D}(\CL{P})$ transversally, in points which do not belong to
$\mathfrak{S}$.

\item[6.] Every irreducible rational nodal degree $d$ curve passing 
through $\CL{P}$ is tangent to $\Gamma$ at most at one
point, with the tangency of the first order. Every point of $\Gamma$
is a point of tangency with at most one such a curve.
\ee

\noindent Define the singular discriminant $\Delta$ as the set of pairs
$(\CL{P},\Gamma)$ that violate the general position
requirements listed above.

\subsection{\textbf{Signs of points and curves}.
 \label{subsec:signscubics}}
Recall that the Welschinger's sign $w_C $ of a rational 
curve $C$ is defined as $w_C=(-1)^{m(C)}$,  where $m(C)$ is the 
mass (i.e., the number of solitary points) of $C$.
For each $p\in\mathfrak{S}$ we define $\iota_p$ by
$$
\iota_p=\begin{cases}
        - W_d +2\cdot\sum_{C\in\Cl{D}(p)}w_C &
        \text{if $p\in\CL{P}$}, \\

        -w_C        & \text{if $p\in\CL{C}_{\Cl{P}}$}, \\

         w_C         & \text{if $p\in\CL{R}_{\Cl{P}}$} 
        \end{cases}
$$
where  $C$ is (the unique) cuspidal or reducible curve
of degree $d$ passing through $\{p\}\cup\CL{P}$ for 
$p\in\CL{C}_{\Cl{P}}\cup\CL{R}_{\Cl{P}}$.

Denote by $\CL{M}_d(\CL{P}, \Gamma)$ the set of real rational 
nodal curves passing through $\CL{P}$ and tangent to $\Gamma$. 
We fix the standard orientation $o_{\rr^2}$ on the plane $\RR^2$
once and for all.
To each $C\in\CL{M}_d(\CL{P}, \Gamma)$ we assign a sign
$\varepsilon_C=w_C\cdot\tau_C$, where $\tau_C$ is a sign of 
tangency of $C$ with $\Gamma$, which is defined similarly to 
Subsection \ref{sub:motivation} as follows: \\
Let $p$ be the point of tangency of $\Gamma$ with $C$.
 For a sufficiently small radius $r$, $C$ divides the disk $\DD(p,r)$
centered at $p$ into two parts. Since the tangency of $\Gamma$ 
and $C$ is of the first order, their quadratic approximations at this point 
$p$ differ. Hence, $\Gamma\cap \DD(p,r)$ belongs to the closure of one of
the two parts of $\DD(p,r)\minus C$. Let $n$ be a normal vector to $C$
at $p$, which looks into  the closure of the part which contains
$\Gamma$, and let $t$ be the tangent vector $t$ to $\Gamma$ at $p$.
Set $\tau_C=+1$ if the frame $(t, n)$ defines the positive
orientation $o_{\rr^2}$ of $\RR^2$, and $\tau_C=-1$ otherwise. See
Figure \ref{fig:Signs of tangency} (compare also with Figure
\ref{fig:morse}b).

\begin{figure}[htb]
\centerline{\includegraphics[width=3.6in]{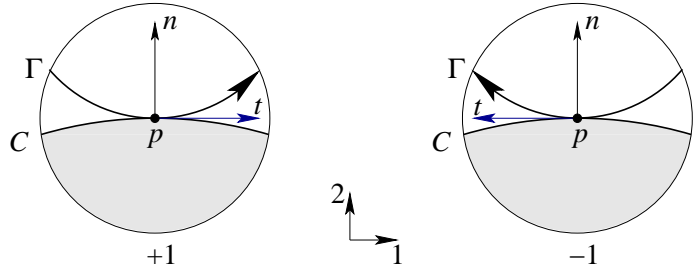}}
\caption{Signs of tangency $\tau_C$.\label{fig:Signs of tangency}}
\end{figure}

\noindent Note that while the immersed curve  $\Gamma$ is oriented, 
the algebraic curve $C$ is not, and we use just the orientation of
$\Gamma$ in order to define the sign $\tau_C$.

\subsection{The statement of the main result.}
Let $N_d(\CL{P}, \Gamma)$ be the algebraic number
$$N_d(\CL{P}, \Gamma):=
\sum\limits_{C\in\Cl{M}_d(\Cl{P}, \Gamma)}\varepsilon_C$$ of real
rational nodal curves passing through $\CL{P}$ and tangent to
$\Gamma$. The main result of this work is the following

\begin{thm}\label{thm:main result}
Let $\CL{P}=\{p_1, \dots ,p_{3d-2}\}\subseteq\RR^2$ and $\Gamma$ be
an immersed oriented curve in $\RR^2$, all in general position. Then

\begin{equation}\label{eqn:main formula}
N_d(\CL{P}, \Gamma)=2\( W_d\cdot
\ind(\Gamma)+\sum\limits_{p\in\mathfrak{S}} \iota_p\cdot
\ind_p(\Gamma)\).
\end{equation}
The number $N_d(\CL{P}, \Gamma)$ is invariant under a regular
homotopy of the pair $(\CL{P}, \Gamma)$ in (each connected component
of) the complement of the singular discriminant $\Delta$.
\end{thm}

\subsection{The case of cubics.}\label{subsection:cubics}
Degree $d=3$ is the first case when all general difficulties appear.
Namely, the number $N_3(\RR)$ is different from one and depends on a
configuration of points, and curves may have cuspidal singularities.
Although cubics have no tacnodes or triple points, these
singularities do not contribute to \eqref{eqn:main formula}, so are
irrelevant for computation of $N_d(\CL{P}, \Gamma)$.
\begin{figure}[thb]
    \centering
    \includegraphics[width=5in]{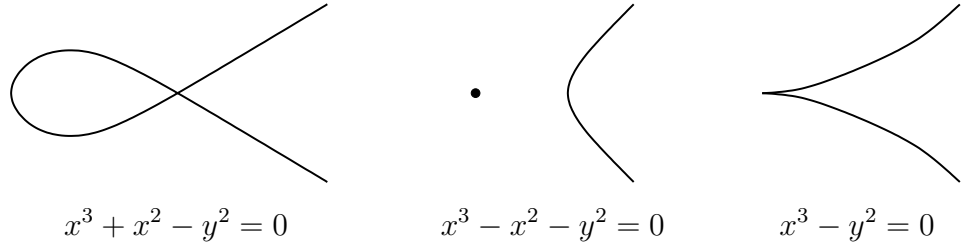}
    $$\quad x^3+x^2-y^2=0 \hspace{0.8in} x^3-x^2-y^2=0 \hspace{0.6in} x^3-y^2=0 $$
    \vspace{-0.3in}
    \caption{A crossing node, a solitary node and a cusp \label{fig:nodal_cubics}}
\end{figure}

There are three types of irreducible real rational cubics: cubics
with one crossing node, cubics with one solitary node and cuspidal
cubics with one cusp point, see Figure \ref{fig:nodal_cubics}. 
A reducible cubic is the union of a line and a conic. 

We count the algebraic number of real rational nodal cubics passing 
through seven generic points and tangent to a generic immersed curve.
The number of curves in $\CL{D}(\CL{P})$ and points in
$\CL{C}_{\Cl{P}}, \CL{R}_{\Cl{P}}$ are bounded from above as
follows. The number of points in $\CL{R}_{\Cl{P}}$ is no more than
$\binom{7}{2}=21$. Due to \cite{K-S}, there are at  most $24$
cuspidal cubics passing through  seven points in general position in
$\CC\PP^2$, hence $|\CL{C}_{\Cl{P}}|\leq 24$. Also, there are
no tacnodes or triple points, so $\CL{T}_{\Cl{P}}=\varnothing$.
From \cite[Theorem $3.2$]{W1} one can deduce that
$|\CL{D}(p)|\in\{0,1\}$. Since $W_3=8$ and for a nodal cubic
$m(C)=0,1$ if $C$ has a crossing or solitary node respectively, we
get
$$
\iota_p=\begin{cases}
        -8+2|\CL{D}(p)| & \text{if $p\in\CL{P}$}, \\
        -1            & \text{if $p\in\CL{C}_{\Cl{P}}$}, \\
       +1           & \text{if $p\in\CL{R}_{\Cl{P}}$}.
        \end{cases}
$$
and Theorem \ref{thm:main result} implies:
\begin{cor}
Let $\CL{P}=\{p_1, \ldots p_7\}\subseteq\RR^2$ and $\Gamma$ be an
immersed oriented curve in $\RR^2$,  all in general position. Then

\begin{equation*}\label{eqn:main formula_cubics}
N_3(\CL{P},\Gamma)=2\big(8\ind(\Gamma)+\sum\limits_{p\in\Cl{P}}\iota_p
\cdot\ind_p(\Gamma)-\sum\limits_{p\in\Cl{C}_{\Cl{P}}}\ind_p(\Gamma)
+\sum\limits_{p\in\Cl{R}_{\Cl{P}}}\ind_p(\Gamma)\big)
\end{equation*}
\end{cor}

\subsection{The main example.}\label{subsection:main example}
Firstly,  consider $\Gamma = T$,  where $T=\partial\DD(p, r)$ is a
circle of infinitesimally small radius $0<r<<1$ in $\RR^2$, centered
at $p$. Suppose that $T$ is far from $\mathfrak{S}$,  i.e,  $T$ is
in the complement of some closed disk $\DD^2$ which contains
$\mathfrak{S}$. See Figure \ref{fig:Special position of points and curves}a.
\begin{figure}[htb]
   \centering
      \includegraphics[height=1.2in]{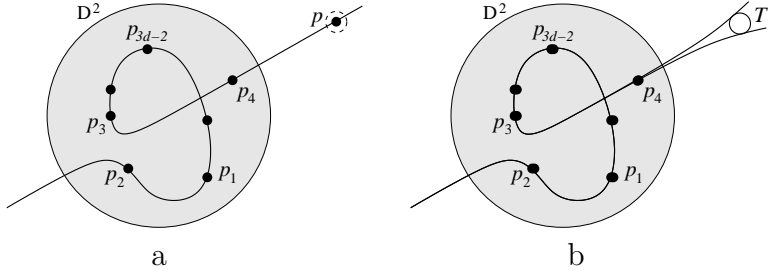}
\vspace{0.1cm}
\centerline{a\hspace{2.1in}b\hspace{0.2in}}
\caption{Changing a point into a circle. 
\label{fig:Special position of points and curves}}
\end{figure}

\noindent Viewing $T$ as a point $p$,  i.e. taking the limit $r\to
0$, we get $3d-1$ generic points $\{p\}\cup\CL{P}$ in the plane. We
have $W_d$ rational nodal curves of degree $d$ passing through
$\{p\}\cup\CL{P}$, counted with their Welschinger's signs. Thus the
algebraic number $N_d(\CL{P}, \Gamma)$ of rational nodal curves
passing through $\CL{P}$ and tangent to $T$ counted with the sign
$\varepsilon_C$ is equal to $2W_d$. Indeed, each rational nodal
curve passing through $\{p\}\cup\CL{P}$ gives $2$ rational nodal
curves passing through $\CL{P}$ and tangent to $T$, see Figure
\ref{fig:Special position of points and curves}b. Moreover, from the
definition of the sign $\tau_C$ we have that $\tau_C=+1$ for any
$C\in \CL{M}_d(\CL{P}, \Gamma)$. Thus in this case
$$N_d(\CL{P}, \Gamma)=\sum\limits_{C\in\Cl{M}_d(\Cl{P}, \Gamma)}
w_C\cdot\tau_C=2\left(\sum\limits_{C\ \text{passes through}\
\{p\}\cup\Cl{P}}w_C\right)=2W_d\ .$$ Reparameterizing the circle $T$
by $\SSS^1\to\SSS^1, \;z\mapsto z^k, \ k\in\ZZ$ (and deforming it
slightly into a general position) we get a curve denoted by $k\cdot
T$ for which we have $\ind(k\cdot T)=k$ and
$$N_d(\CL{P}, \Gamma)=2k\, W_d\ .$$ Since every immersed curve $\Gamma$ is
homotopic in the class of immersions in $\RR^2$ to $k\cdot T$, where
$k=\ind(\Gamma)$, we have that $N_d(\CL{P}, \Gamma)=2W_d\cdot
\ind(\Gamma)$ for a curve $\Gamma$ lying in the complement of some
closed disk of a sufficiently large radius, which contains
$\mathfrak{S}$.

\subsection{The idea of the proof.}
Consider a  solid torus $M=\DD^2\times\SSS^1$,  where $\DD^2$ is a
sufficiently large closed disk containing $\mathfrak{S}$. We will
show that the number $N_d(\CL{P}, \Gamma)$ in  Theorem \ref{thm:main
result} is the intersection number $I(L, \cl{\Sigma};M)$ of an
oriented smooth curve $L$ with a compactification $\cl{\Sigma}$ of
an open two-dimensional surface $\Sigma$ in $M$. The surface
$\Sigma$ is
constructed as follows:\\
For each $p\in\DD^2\minus \mathfrak{S}$, we use a contact element
(line) of curves passing through $\{p\}\cup\CL{P}$ to get $\Sigma$
as a lift of $\DD^2\minus \mathfrak{S}$ into $M$. Lifting $\Gamma$
into $M$ in a similar way we get $L$. The Welschinger's sign $w_C$
gives rise to the orientation on $\Sigma$ and the
orientation of $\Gamma$ defines the orientation of $L$.\\
In order to define the intersection number, we compactify $\Sigma$ to
get a compact surface $\cl{\Sigma}$ with boundary. This is done by
blowing up punctures $\mathfrak{S}$ on $\DD^2$,  i.e.,  we cut out a
small open disk around each puncture and then we lift the remaining
domain into $M$. Due to generality of a pair $( \CL{P}, \Gamma)$,
$L$ transversally intersects $\cl{\Sigma}$ in a finite number of regular 
points of $\cl{\Sigma}$.
Each point $(p, \xi)\in L\pitchfork\cl{\Sigma}$ corresponds to a curve 
passing through $\CL{P}$ and tangent to $\Gamma$. We prove that 
the local intersection number $I_{(p, \xi)}(L,\cl{\Sigma};M)$ equals 
to $\tau_C\cdot w_C$, and thus
$$N_d(\CL{P}, \Gamma)=I(L,\cl{\Sigma};M).$$
Now to get the right hand side of the formula \eqref{eqn:main
formula} we use the homological interpretation of the intersection
number. We take $\Gamma':=\ind(\Gamma)\cdot T$ as in the main
example, see Subsection \ref{subsection:main example},
so $\Gamma'$ is homotopic to $\Gamma$ in the class of
immersions. Hence $[\Gamma]-[\Gamma']=\partial K$ in
$C_1(\DD^2;\ZZ)$ for some 2-chain $K\in C_2(\DD^2;\ZZ)$. Then for
the lifts $L'$ and $\CL{K}$ of $\Gamma'$ and $K$,
respectively, into $M$ we have $[L]-[L']=\partial\CL{K}$ in
$C_1(M;\ZZ)$, and hence
$$I(L, \cl{\Sigma};M)=I(L', \cl{\Sigma};M)+I(\partial\CL{K}, \cl{\Sigma};M).$$
From the main example we obtain
$I(L', \cl{\Sigma};M)=2W_d\cdot \ind(\Gamma')=2W_d\cdot \ind(\Gamma).$
Finally, to complete the proof  we show that
$$I(\partial\CL{K}, \cl{\Sigma};M)=I(\CL{K}, \partial\cl{\Sigma};M)=
2\sum\limits_{p\in\mathfrak{S}}\iota_p
\cdot \ind_p(\Gamma).$$

\begin{rem}\label{rem:sigma4lines}
A simple way to visualize the surface $\Sigma$ is to apply the above
construction to the model example of Section \ref{sub:motivation}.
In this case $\mathfrak{S}$ consists of one point $p$ and the
contact element of any line passing through $p$ is the line itself, so 
the surface is a helicoid, see
Figure \ref{fig:Compactification}.
\end{rem}

\section{The proof of the main result.}

The manifold of oriented contact elements (directions) of the plane
is $\ST^*\RR^2$, the spherization of the cotangent bundle of the
plane. We fix an orientation $o_{\St^*\rr^2}=o_{\rr^2}\times
o_{\SSS^1}$ on $\ST^*\RR^2$, where $o_{\SSS^1}$ is the standard
counterclockwise orientation on $\SSS^1$.

\subsection{Construction of $M, \Sigma, \cl{\Sigma},  L$.}
\paragraph{\textbf{Construction of $\Sigma$}.}
Consider a $(3d-2)$-tuple $\CL{P}=\{p_1,\ldots, p_{3d-2}\}$ of
points in $\RR^2$ in general position. Recall that
$\mathfrak{S}:=\CL{P}\cup\CL{C}_{\Cl{P}}\cup\CL{R}_{\Cl{P}}$,
see Subsection \ref{subsect:generalposition}. Let
$S:=\RR^2\minus \mathfrak{S}$. 
Define $\Sigma\subset \ST^*(\RR^2\minus \mathfrak{S})$ by
\begin{equation}\nonumber\begin{aligned}
&\Sigma:=\left\{(p,\xi)\ \vline
\begin{array}{l}
\text{there is a rational degree $d$ curve passing through }\\ 
\{p\}\cup\CL{P} \text{ with $\xi$ as a tangent direction at a point}\  p
\end{array}
\right\}
\end{aligned}\end{equation}
Denote by $\pi:\Sigma\to S$ the natural projection $\pi((p, \xi))=p$.

\begin{prop}\label{prop:Sigma}
The set $\Sigma$ is a (non-compact) immersed orientable two-dimensional 
surface in $\ST^*\RR^2$. 

\end{prop}

\begin{proof}
Consider an arbitrary $p\in S$ and choose a branch of a curve 
$C_0$ passing through $p$. Lifting the point $p$ using the tangent 
direction $\xi$ of this branch, we get a point in $\PT^*\RR^2$ 
which gives a pair $(p, \pm\xi)$ of points in the double 
covering $\ST^*\RR^2$ of $\PT^*\RR^2$.
Include the given curve $C_0$ into a smooth 1-parametric family 
$C_t$, $t\in(-\eps,\eps)$, $\eps<<1$ of rational curves of degree 
$d$ passing through $\CL{P}$. 
Since $p\notin\mathfrak{S}$, in a small neighborhood $U$ 
of $p$ the corresponding family of contact elements smoothly 
depends on the point of contact. 
The lift of this family of contact elements to $\Sigma$ gives 
a smooth leaf $\Sigma(C_t)$ of $\Sigma$ in a neighborhood 
of $(p,\xi)$. 
The topological structure of this smooth leaf $\Sigma(C_t)$ 
depends on the type of the curve $C_0$.
  
If the initial curve $C_0$ is not cuspidal, the family $C_t$ foliates 
$U$, so the projection $\pi$ of $\Sigma(C_t)$ on $U$ is a diffeomorphism.
See Figure \ref{fig:surface}a.
Note that the Welschinger's sign $w$ of all curves $C_t$ in the family 
is the same. This allows us to define a local orientation of $\Sigma(C_t)$ 
as follows.
It suffices to define a continuous field $\nu$ normal to $\Sigma(C_t)$. 
Since  $T_x\Sigma(C_t)\pitchfork T_x\SSS^1$ at each point $x\in\Sigma(C_t)$, 
such a normal vector $\nu_x$ is determined by its projection to
$T_x\SSS^1$. Recall, that we have already fixed the
orientation $o_{\SSS^1}$ on the fiber $\SSS^1$ of $\ST^*\RR^2$. 
We set the direction of $T_x\SSS^1$-component of $\nu_x$ in the 
direction of $o_{\SSS^1}$ if $w=+1$, and opposite to this
direction if $w=-1$.  

\begin{figure}[ht]
    \centering
    \includegraphics[width=5in]{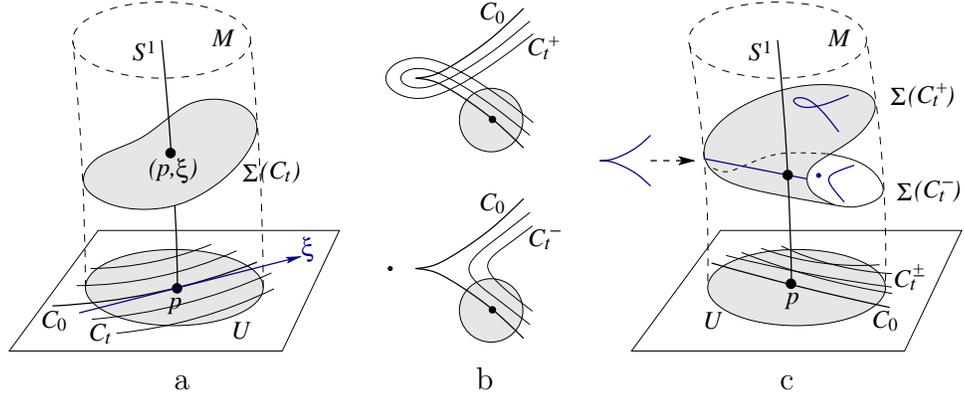}
\vspace{-0.1cm}
\centerline{a\hspace{1.5in}b\hspace{1.5in}c}
\vspace{-0.3cm}
\caption{Smooth sheets and folds of $\Sigma$.\label{fig:surface}}
\end{figure}

If $C_0$ is cuspidal, both subfamilies $C^-=\{C_t| t\in(-\eps,0)\}$ 
and $C^+=\{C_t| t\in(0,\eps)\}$ foliate the same region of 
$U\minus C_0$, with Welschinger's sign $w^\pm$ of all curves in 
each subfamily $C^\pm$ being the same and $w^-=-w^+$, 
see \cite[Proposition 2.16]{W1}. 
The corresponding leaf $\Sigma(C_t)$ of $\Sigma$ has the structure 
of a fold, see Figure \ref{fig:surface}b. 
Two (open) sheets $\Sigma^\pm(C_t)$ of this fold correspond to the 
lift of contact elements of curves in the subfamilies $C^\pm$. Since 
curves in the subfamilies $C^\pm$ have opposite Welschinger's signs 
$w^\pm$, we define the $T_x\SSS^1$-component of $\nu_x$ as 
above for $x\in\Sigma^\pm(C_t)$ and extend $\nu$ continuously 
(with $T_x\SSS^1$-component being zero) to the fold of $\Sigma(C_t)$
(i.e. the lift of $C_0$).
\end{proof}

The topological structure of $\Sigma$ in a small neighborhood of 
$(p, \xi)$ depends on the type of the point $p$. Namely, if $p$ is generic
then $3d-1$ points $\{p\}\cup\CL{P}$ are in general position and 
define a real general $\frac{(d-1)(d-2)}{2}$-dimensional space of 
curves of degree $d$ that contains $N_d(\RR)$ irreducible rational 
nodal curves which intersect transversely in $p$. All of these 
curves pass through $p$ with different tangent directions. Thus 
above a small neighborhood of a generic point $p$ the surface 
$\Sigma$ has a structure of a smooth $2N_d(\RR)$-covering.
The same covering structure appears if $p$ lies on a curve 
with a triple point or tacnode (but is different from a tacnode, 
$p\notin\CL{T}_{\Cl{P}}$). 
Note that while the number $2N_d(\RR)$ of sheets over $p$ depends 
on $\CL{P}$ and $p$, the  number of sheets counted with their 
orientations (i.e., the local degree of $\pi:\Sigma\to S$ at a regular 
value $p$) does not depend on $\{p\}\cup\CL{P}$ and equals $2W_d$.

If some branches of curves which pass through $\{p\}\cup\CL{P}$ 
have the same tangent directions in $p$ (in particular, if $p$ is a 
tacnode), the corresponding lift has the structure of an open book, 
and sheets of the book come in pairs with each pair forming a smooth 
surface, see Figure \ref{fig:singularities of surface}a. The same open 
book structure appears if $p$ lies on a reducible curve or a curve with 
a node at some $p_i\in\CL{P}$, see Figure  \ref{fig:singularities of surface}b.

\begin{figure}[ht]
    \centering
    \includegraphics[width=5in]{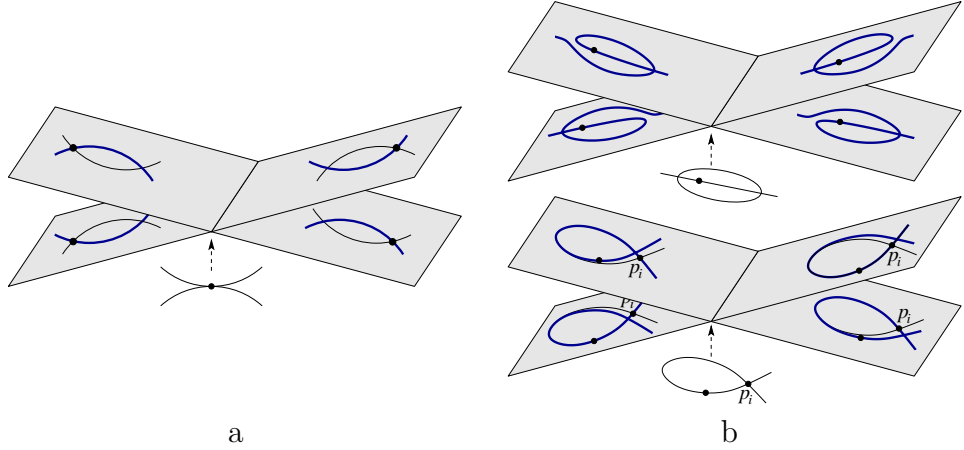}
\vspace{-0.1cm}
\centerline{a\hspace{2.5in}b}
\vspace{-0.3cm}
\caption{Branch points of $\Sigma$.\label{fig:singularities of surface}}
\end{figure}

\begin{rem}
We have to cut out points of $\mathfrak{S}$ from $\RR^2$ in
the construction of $\Sigma$ above since $\Sigma$ does not 
extend to an immersed surface over $\mathfrak{S}$. 
Indeed, if $p\in\CL{C}_{\Cl{P}}\cup\CL{R}_{\Cl{P}}$, tangent 
directions to curves in 1-parametric families $C_t$ used in the 
proof of Proposition \ref{prop:Sigma} do not change smoothly 
in a neighborhood of $p$ (see Figure \ref{fig:fields1}b,c). 
If $p\in\CL{P}$, the obstacle is different: there are infinitely 
many tangent directions of curves passing through $p$.
Note, that while points in $\CL{T}_{\Cl{P}}$ are also singular 
points of curves in $\CL{D}(\CL{P})$, there is no need to cut them 
out from $\RR^2$ since the corresponding tangent directions to 
curves in $C_t$ change smoothly in a neighborhood of a tacnode 
or a triple point, and the number of tangent directions in $p$ is finite. 
\end{rem}

\paragraph{\textbf{Compactification of $\ST^*\RR^2$ and $\Sigma$.}
\label{subsect:compactification}}

In order to use the intersection theory,  we need to compactify both
the open manifold $\ST^*\RR^2$, and the non-compact surface $\Sigma$
with punctures over $\mathfrak{S}$.\\

Let $\DD^2:=\cl{\DD(0, R)},  R>>1$ be a closed disk in the plane
$\RR^2$, centered at the origin and of a sufficiently large radius, 
such that $\mathfrak{S}\subseteq \cl{\DD(0,R/2)}\subseteq\DD^2$. 
Define $M:=\ST^*\DD^2=\DD^2\times\SSS^1.$

Let us choose $0<\delta<<1$ sufficiently small, such that

\be
\item[1.]$\cl{\DD(p, \delta)}\cap\cl{\DD(q, \delta)}=\varnothing$ for all
$p\neq q\in \mathfrak{S}$,
\item[2.]$\cl{\DD(p, \delta)}\cap\Gamma=\cl{\DD(p, \delta)}\cap\partial\DD^2=
\varnothing$ for all $p\in \mathfrak{S}$,
\item[3.]$\cl{\DD(p, \delta)}$ does not contain points but $p$ of mutual 
intersections of all curves  from $\CL{D}(\CL{P})$  for all $p\in \mathfrak{S}$,
\item[4.]$\partial\cl{\DD(p, \delta)}$ intersects transversally every 
$C\in\CL{D}(\CL{P})$ for all $p\in \mathfrak{S}$. These intersections
look as shown in Figure \ref{fig:blowup disks int}.
\ee
\begin{figure}[htb]
    \includegraphics[width=5in]{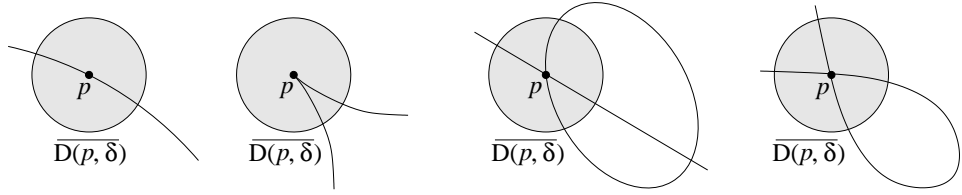}
\caption{Intersections of blowup disks with cubics.\label{fig:blowup disks int}}
\end{figure}

For each $p\in\mathfrak{S}$ we cut out the disk $\DD(p, \delta)$ from
$\DD^2$
and define
$$\bar{S}:=\DD^2\minus\bigcup_{p\in \mathfrak{S}}\DD(p, \delta), \qquad \cl{\Sigma}:=\pi^{-1}(\bar{S})=\Sigma\cap\(\bar{S}\times\SSS^1\).$$
%
%

For all $p\in \mathfrak{S}$ let $\sigma_p:=\Sigma\cap
(\partial\DD(p, \delta)\times\SSS^1)$; it is the union of several
smooth closed simple curves on $\Sigma$. We equip
$\sigma_p$ with the orientation induced from $\Sigma$.
See Figure \ref{fig:Compactification}.

\begin{figure}[ht]
\centerline{\includegraphics[height=1.7in]{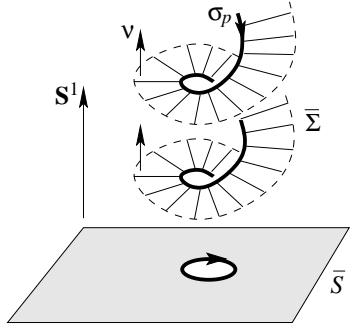}}
\caption{A compactification $\cl{\Sigma}$ of $\Sigma$.
\label{fig:Compactification}}
\end{figure}

\paragraph{\textbf{Construction of $L$}.}
Let $\Gamma = f(\SSS^1)$ be oriented immersed curve,  where
$f:\SSS^1\looparrowright\RR^2$ is an immersion. Choosing the unit
tangent vector to $f(t)$ as the contact element, we get a lift $L$
of $\Gamma$ into $\ST^*\RR^2$:
$$L:=F(\SSS^1), \qquad F:\SSS^1\hookrightarrow \ST^*\RR^2, \qquad t\mapsto\left(f(t),
\frac{f'(t)}{\|f'(t)\|}\right).$$ It follows that $L$ is an oriented
closed one-dimensional submanifold of $\ST^*\RR^2$. 
If $\Gamma$ is generic in the sense of Section \ref{subsect:generalposition}, 
then $L$ intersects $\Sigma$ only in regular points (i.e., points such that 
$\Sigma\cap U$ is diffeomorphic to $\RR^2$ for some neighborhood $U$ 
in $M$).

\subsection{Two ways to calculate the intersection number 
$I(L, \cl{\Sigma};M)$.}
We will calculate the intersection number $I(L, \cl{\Sigma};M)$ in two 
different ways, which will correspond to the LHS and the RHS of
equality \eqref{eqn:main formula}.

\paragraph{\textbf{The intersection number 
$I(L, \cl{\Sigma};M)$ via the algebraic number $N$.}}

Every point $(p, \xi)\in \cl{\Sigma}\cap L$ corresponds to an irreducible
rational nodal curve $C(p, \xi)$,  passing through $\CL{P}$ and tangent to
$\Gamma$ at the point $p$ with the tangent direction $\xi$.
Since $\CL{P}$ and $\Gamma$ are in general position,  we have that $L$
and $\cl{\Sigma}$ intersect transversally,  and since $\dim(L)=1$,
$\dim(\cl{\Sigma})=2$ and $\dim(M)=3$,  we have that
$\dim(L\pitchfork \cl{\Sigma})=0$. So the number of points in
$L\pitchfork\cl{\Sigma}$ is finite. Both $L$ and $\cl{\Sigma}$ are
oriented, as is $M$,  hence the intersection number $I(L, \cl{\Sigma};M)$
is well defined and we have
$$I(L, \cl{\Sigma};M)=\sum_{(p, \xi)\in L\pitchfork\cl{\Sigma}}I_{(p, \xi)}(L, \cl{\Sigma};M), $$
where $I_{(p, \xi)}(L, \cl{\Sigma};M)$ is the local intersection number.

\begin{prop}
For every $(p, \xi)\in L\pitchfork\cl{\Sigma}$ we have
$$I_{(p, \xi)}(L, \cl{\Sigma};M)=\varepsilon_{C(p, \xi)}, $$ where
$\varepsilon_C$ is the sign of the curve $C$,  see Subsection
\ref{subsec:signscubics}.
\end{prop}

\begin{proof}
The orientation of $T_{(p, \xi)}\cl{\Sigma}$ is defined by the
Welschinger's sign $w_{C(p, \xi)}$. The curve $L$
intersects $\cl{\Sigma}$ in the direction of the oriented fiber
$F$ iff $\tau_{C(p, \xi)}=+1$, see Figure \ref{fig:tauorientation}.
\begin{figure}[ht]
\centerline{\includegraphics[width=4.1in]{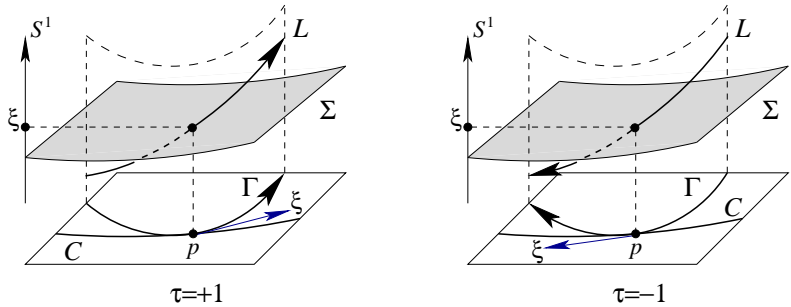}}
\caption{Intersection of $L$ with $\Sigma$.\label{fig:tauorientation}}
\end{figure}
Hence the orientation of $T_{(p, \xi)}\cl{\Sigma}\oplus T_{(p,\xi)}L$ 
differs from the orientation $o_{\St^*\rr^2}\big|_{M}$ of $M$ by the 
sign $\varepsilon_{C(p, \xi)}=w_{C(p,\xi)}\cdot\tau_{C(p, \xi)}$ and 
we get $I_{(p, \xi)}(L,\cl{\Sigma};M)=\varepsilon_{C(p, \xi)}.$
\end{proof}

\begin{cor}
We have \quad $\displaystyle{N_d(\CL{P},
\Gamma)=\sum\limits_{C\in\Cl{M}_d(\Cl{P}, \Gamma)}\varepsilon_C=I(L,
\cl{\Sigma};M)}$
\end{cor}

\paragraph{\textbf{The intersection number 
$I(L, \cl{\Sigma};M)$ via a homological theory.}
\label{subsect:intersectionnumberviahomology}}
Let us take $k\cdot T, \ k=\ind(\Gamma)$ as in Subsection
\ref{subsection:main example} which is regularly homotopic to
$\Gamma$ in $\DD^2$,  and $h:\SSS^1\times [0, 1]\to\DD^2$ be a
homotopy between $k\cdot T$ and $\Gamma$. Denote
$\Gamma_t:=h(\SSS^1\times\{t\}), \ t\in [0, 1]$, so $\Gamma_0=k\cdot
T$ and $\Gamma_1=\Gamma$. Denote by $L_t$ a lift of $\Gamma_t$ to
$M$. Then $L'=L_0, L=L_1$ and  a 2-chain
$\CL{K}:=\{L_t|t\in[0,1]\}\in C_2(M;\ZZ)$ realizes a homotopy
between $L'$ and $L$. We choose an orientation of $\CL{K}$ such that
$\partial\CL{K}=[L]-[L']$. Because of the homotopy invariance of the
intersection number, we may choose a special homotopy $h$ as
follows. For all $p\in\mathfrak{S}$ pick an open neighborhood $U_p$
of $\cl{\DD(p, \delta)}$ and a direction $\xi_p$ transversal to
tangent directions of all curves from $\CL{D}(\CL{P})$ at $p$. See
Figure \ref{fig:flatgamma}a.
\begin{figure}[htb]
   \centering
   \includegraphics[height=2.0in]{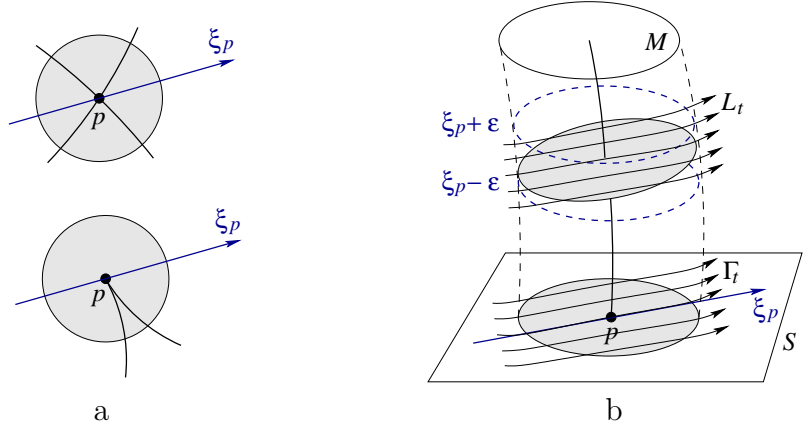}
\vspace{0.1cm}
\centerline{a\hspace{2.6in}b\hspace{0.5in}}
\caption{A flat homotopy of $\Gamma$.\label{fig:flatgamma}}
\end{figure}
Now, choose the homotopy $h$ so that for all $t\in [0, 1]$ with
$\Gamma_t\cap U_p\neq\varnothing$, the fragment $\Gamma_t\cap U_p$
is close to a straight interval in the direction $\xi_p$.
For such a homotopy the part $\CL{K}\cap(U_p\times\SSS^1)$ of 
$\CL{K}$ over $U_p$ is almost flat,  i.e.,  lies in a thin cylinder
$$\CL{K}\cap(U_p\times\SSS^1)\subseteq U_p\times
(\xi_p-\eps, \xi_p+\eps),$$ for some small $0<\eps<<1$.
See Figure \ref{fig:flatgamma}b.

By the additivity of the intersection number and according to the 
calculations in the Subsection \ref{subsection:main example} we have
\begin{multline*}
I(L, \cl{\Sigma};M)=I(L', \cl{\Sigma};M)+I(\partial\CL{K}, \cl{\Sigma};M)= \\
=2W_d\cdot \ind(k\cdot T)+I(\partial\CL{K}, \cl{\Sigma};M)=2W_d\cdot
\ind(\Gamma)+I(\partial\CL{K}, \cl{\Sigma};M)
\end{multline*}
It remains to compute $I(\partial\CL{K},\cl{\Sigma};M)$.

\begin{lem}
For $\cl{\Sigma}, \CL{K}, \sigma_p$ and $L$ as before we have
$$I(\partial\CL{K},\cl{\Sigma};M)=\sum\limits_{p\in \mathfrak{S}}I(\CL{K},\sigma_p;M).$$
\end{lem}

\begin{proof}
Recall  that $I( \partial\CL{K},\cl{\Sigma};M)=I(\CL{K},\partial\cl{\Sigma};M)$.
Now,  as a 1-chain in $M$, $$\partial\cl{\Sigma}=
\partial\cl{\Sigma}\cap(\partial\DD^2\times\SSS^1))+\sum_{p\in \mathfrak{S}}\sigma_p.$$
Since $\CL{K}\cap(\partial\DD^2\times\SSS^1)=\varnothing$,  we get
$I(\CL{K},\partial\cl{\Sigma};M)=\sum\limits_{p\in \mathfrak{S}}I(\CL{K},\sigma_p;M)$.
\end{proof}

The following proposition completes the proof of the main theorem.

\begin{prop}
For every $p\in\mathfrak{S}$ we have
$$I(\CL{K},\sigma_p;M)=2\iota_p\cdot \ind_p(\Gamma).$$
\end{prop}

\begin{proof}
Firstly, recall that $\ST^*\DD^2\to \PT^*\DD^2$ is a $2$-fold
covering, so for every component of the lift of
$\partial\DD(p,\delta)$ to $\PT^*\DD^2$ there are two components in
$\ST^*\DD^2$, which explains the coefficient $2$ in the RHS.
Secondly, note that $I(\CL{K},\sigma_p;M)=I(\CL{K}_p,\sigma_p;M)$
for every $p\in\mathfrak{S}$, where
$\CL{K}_p:=\CL{K}\cap(U_p\times\SSS^1).$
\\
In order to compute $I(\CL{K}_p,\sigma_p;M)$ we study the homology 
class $[\sigma_p]\in\mathrm{H}_1(M;\ZZ)$ of $\sigma_p$. 
Since $\mathrm{H}_1(M;\ZZ)=\ZZ\langle [F]\rangle$, where $[F]$ 
is the class of the fiber, we conclude that $[\sigma_p]=k_p\cdot [F]$ 
for some $k_p\in\ZZ$. 
The number $k_p$ is the degree $\deg G_p$ of the
corresponding projection map $G_p:\sigma_p\to F$ to the 
fiber $F$ of $M$. It can be computed as the algebraic number of 
preimages  $(G_p)^{-1}(\xi)$ of a regular value $\xi$.
Each preimage $(q,\xi)\in\sigma_p$ is counted with its 
sign -- the local degree $\deg_{(q,\xi)}G_p$ of $G_p$ at $(q,\xi)$.

Both preimages and their signs can be recovered from the 
projection $\pi:\sigma_p\to\partial\DD(p,\delta)$.
Indeed, a preimage $(q,\xi)\in\sigma_p$ corresponds to the point 
$q\in\partial\DD(p,\delta)$ and the curve $C(q, \xi)$ passing
through $q\cup\CL{P}$ and having a tangent direction $\xi$ at $q$.
Moreover, the orientation of $T_{(q, \xi)}\sigma_p$ is induced from 
that of $\Sigma$ which, in turn, is defined by the Welschinger's 
sign $w_{C(q, \xi)}$.    
Thus the orientation of the projection $\pi:\sigma_p\to\partial\DD(p,\delta)$ 
at $q$ differs from the clockwise orientation on $\partial\DD(p,\delta)$  
by $w_{C(q, \xi)}$ (see Figure \ref{fig:Compactification}) . 
Therefore, the local degree $\deg_{(q,\xi)}G_p$ 
equals to $w_{C(q, \xi)}\cdot\rho_q$, where $\rho_q=1$  
(resp. $\rho_q=-1$) if the field of tangent directions of curves 
corresponding to $\sigma_p$ rotates counter-clockwise (resp. 
clockwise) w.r.t. $\xi$ as we move clockwise along $\partial\DD(p,\delta)$ 
in a neighborhood of $q$. 

To find the corresponding curves $C(q, \xi)$, note that any curve
passing through $q\cup\CL{P}$ for $q\in\partial\DD(p,\delta)$ is
obtained by a small deformation of some rational curve $C_p$ of 
degree $d$ in the following finite set:
\newpage
\be
\item[(i)] If $p\notin\CL{P}$, then $C_p$ passes through $3d-1$ 
              points $p\cup\CL{P}$.
\item[(ii)] If $p\in\CL{P}$, then $C_p$ passes through $3d-2$ points $\CL{P}$ 
and either has a node at $p$, or has a tangent direction $\xi$ at $p$.
\ee
Consider these cases separately using the standard methods of 
singularity theory.

\noindent{\bf Case 1: $p\notin\CL{P}$.} If $C_p$ is nodal, its small 
deformation is shown in Figure \ref{fig:fields1}a.  For sufficiently small 
$\delta$, the corresponding tangent field is almost constant and for a
generic choice of $\xi$ there are no preimages. 

If $C_p\in\CL{D}(\CL{P})$, its small deformations for $p\in\CL{C}_{\Cl{P}}$,  
and $p\in\CL{R}_{\Cl{P}}$ are shown in Figures
\ref{fig:fields1}b and \ref{fig:fields1}c respectively. 

\begin{figure}[htb]
    \includegraphics[width=5.0in]{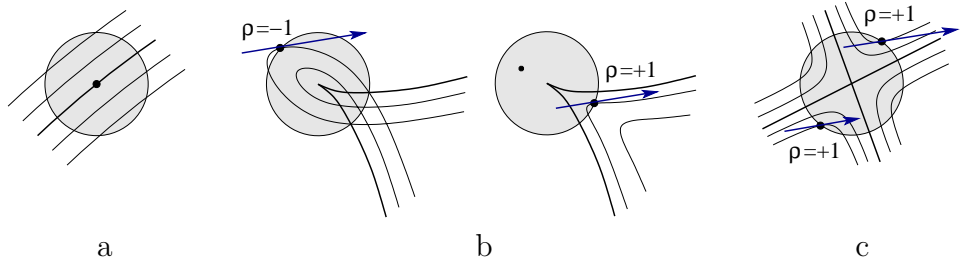}
\centerline{a\hspace{1.9in}b\hspace{1.9in}c}
\vspace{-0.1cm}
\caption{Counting preimages for $p\notin\CL{P}$.
\label{fig:fields1}}
\end{figure}

For $p\in\CL{R}_{\Cl{P}}$, there are two preimages, both with 
$\rho_q=1$ and $w_{C(q,\xi)}=w_{C_p}$, see Figure \ref{fig:fields1}c. Thus 
the local degree of each of these two preimages equals $w_{C_p}=\iota_p$,
so $\deg G_p=2\iota_p$. 

For $p\in\CL{C}_{\Cl{P}}$ there are also two preimages: one with 
$\rho_q=-1$ and $w_{C(q,\xi)}=w_{C_p}$, and the other with 
$\rho_q=1$ and $w_{C(q,\xi)}=-w_{C_p}$, see Figure \ref{fig:fields1}b. 
Thus the local degree of each of these two preimages equals $-w_{C_p}=\iota_p$
and $\deg G_p=2\iota_p$.

\noindent{\bf Case 2: $p\in\CL{P}$.}
If $C_p$ is nodal with node at $p$ (i.e., $C_p\in\CL{D}(p)$), its small 
deformations are shown in Figure \ref{fig:fields2}a.  Again, for sufficiently 
small $\delta$, the corresponding tangent fields are almost constant and for 
a generic choice of $\xi$ there are no preimages. 

\begin{figure}[htb]
    \includegraphics[width=4.4in]{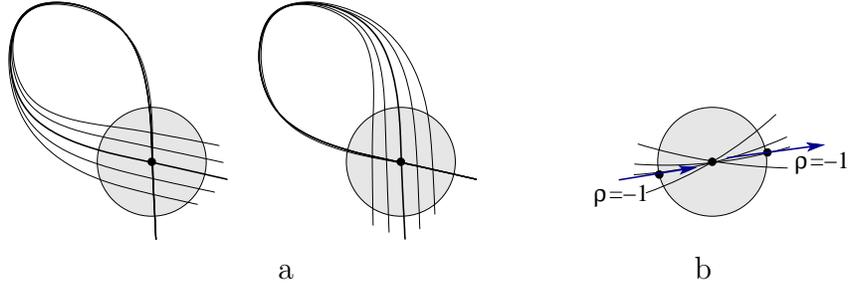}
\vspace{-0.1cm}
\centerline{\hspace{0.7in}a\hspace{2.1in}b}
\vspace{-0.3cm}
\caption{Counting preimages for $p\in\CL{P}$.
\label{fig:fields2}}
\end{figure}

If $C_p=C(p,\xi)$ is nodal with a tangent direction $\xi$ at $p$, its 
small deformations are shown in Figure \ref{fig:fields2}b.  There are 
two preimages, both with $\rho_q=-1$ and $w_{C(q,\xi)}=w_{C(p,\xi)}$, 
see Figure \ref{fig:fields2}b. 
Thus the local degree of each of these two preimages equals $-w_{C(p,\xi)}$
and each such curve $C(q,\xi)$ contributes $-2w_{C(p,\xi)}$ to $\deg G_p$.
By \cite[Proposition 3.4]{W1} 
$$\sum_{C(p,\xi)} w_{C(p,\xi)}+2\sum_{C\in\Cl{D}(p)}w_C=W_d\,,$$
where the first sum is over all nodal curves with a tangent direction 
$\xi$ at $p$. Therefore,  in this case we also get  
$$\deg G_p=-2\sum_{C(p,\xi)} w_{C(p,\xi)}=
2(-W_d+2\sum_{C\in\Cl{D}(p)}w_C)=2\iota_p\,.$$

We finally conclude that in all cases $k_p=\deg G_p=2\iota_p$.
By the choice of the homotopy $h$,
$$I(\CL{K}_p,\sigma_p;M)=k_p\cdot I(\CL{K}_p,\{p\}\times\SSS^1;M)=
2\iota_p\cdot I(\CL{K}_p,\{p\}\times\SSS^1;M).$$

We finish the proof by observing that
\begin{multline*}
I(\CL{K}_p,\{p\}\times\SSS^1;M)=
I(h(\SSS^1\times[0, 1]),[p];\RR^2)=\\
=I(h(\SSS^1\times[0, 1]),[p]-[\infty];\RR^2)=\\
=-I(\partial h(\SSS^1\times[0, 1]),[p,\infty];\RR^2)=\ind_p(\Gamma).
\end{multline*}
\end{proof}

\section{Finite type invariants.}\label{sec:FTI}

Finite type invariants generalize polynomial functions.
This notion is based on the following classical theorem:

\begin{thm}[Frechet 1912]
Given  $x_0,  x_1^\pm,\dots , x_n^\pm \in\RR$ and an $n$-tuple
$\eps=(\eps_1,\dots,\eps_n)\in\{ -1,1\}^n$, let
$x_{\eps}=x_0+x_1^{\eps_1}+\dots+x_n^{\eps_n}$ and
$|\eps|=\prod_{i=1}^n\eps_i$. Then $C^0$-function $f:\RR\to\RR$ is a
polynomial of degree less than $n$, iff
$$\displaystyle\sum_{\eps\in\{ -1,1\}^n}(-1)^{|\eps|}f(x_\eps)=0$$
for any choice of $x_0$ and $x_1^\pm,\dots , x_n^\pm$.
\end{thm}

Finite type invariants are topological analogues of this definition.
Corresponding theories are developed for a variety of objects:
knots, 3-manifolds, plane curves, graphs, etc. (see \cite{M-P} for a
general theory of finite type invariants of cubic complexes). Let us
briefly recall the main notions in the case of immersed curves in a
punctured plane. Let $\mathfrak{S}\subset\RR^2$ be a finite set
of marked points and $\Gamma_{sing}$  be an immersed plane
curve with $n$ non-generic fragments, contained in $n$ small disks
$\DD_i$ (all in general position). Fix an arbitrary pair of resolutions 
for each $\DD_i$ and call one of them positive and the other negative 
(again, arbitrarily).
Here by a resolution of $\Gamma_{sing}$ in a disk $\DD_i$ we
mean a homotopy of $\Gamma_{sing}$ inside $\DD_i$, fixed on the
boundary $\partial\DD_i$, so that the resulting curve is a generic
immersion inside $\DD_i$ and does not pass through
$\mathfrak{S}\cap\DD_i$. See Figure \ref{fig:fti}.
\begin{figure}[htb]
\centerline{\includegraphics[height=1.7in]{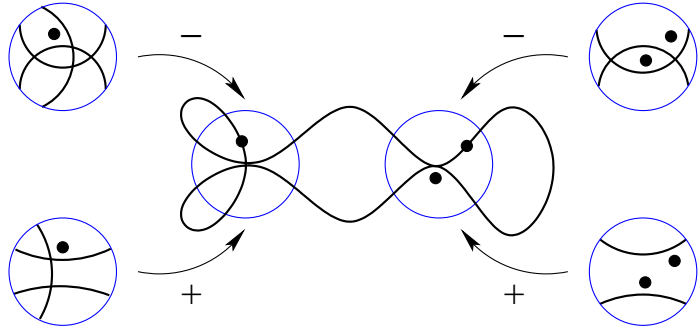}}
\caption{A non-generic curve with a pair of resolutions in each disk.
\label{fig:fti}}
\end{figure}

For an $n$-tuple $\eps\in\{ -1,1\}^n$, resolve all singularities of
$\Gamma_{sing}$ choosing the corresponding $\eps_i$ resolution in
each disk $\DD_i$. Denote by $\Gamma_\eps$ the resulting curve. In
this way, as $\eps$ runs over $\{-1,1\}^n$, we obtain $2^n$
generically immersed curves $\Gamma_\eps$. See Figure \ref{fig:fti_res}.
\begin{figure}[htb]
\centerline{\includegraphics[height=1.6in]{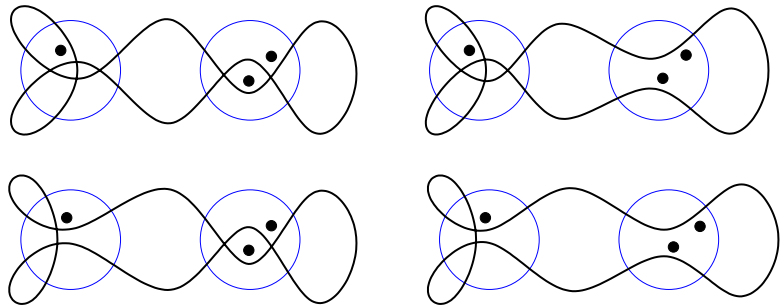}}
\caption{Resolved generic curves.
\label{fig:fti_res}}
\end{figure}

Denote
$|\eps|=\prod_{i=1}^n\eps_i$. A locally-constant function $f$ on the
space of generically immersed curves is called an invariant of
degree less than $n$, if
$$\sum_{\eps\in\{ -1,1\}^n}(-1)^{|\eps|}f(\Gamma_\eps)=0,$$
for any choice of the curve $\Gamma_{sing}$ and its resolutions.

When $\mathfrak{S}=\varnothing$, the only invariant of degree zero
(i.e., a constant function on the space of immersed curves) is the
rotation number $\ind(\Gamma)$. Various interesting invariants of
degree one for $\mathfrak{S}=\varnothing$ were extensively
studied by V.~Arnold, see \cite{Arn}. When $\mathfrak{S}$
consists of one point, we get an additional simple invariant of
degree one, namely $\ind_p(\Gamma)$. In a general case, any
linear combination of $\ind(\Gamma)$ and $\ind_p(\Gamma)$, 
$p\in\mathfrak{S}$ is an invariant of degree at most one.

Finite type invariants naturally appear in real enumerative geometry.
One of the simplest examples was considered in Section \ref{sub:motivation}.
Note that in the formula \eqref{eq:N motivation}, an algebraic
number of lines passing  through a point $p$ and tangent to a
generic immersed curve $\Gamma\subset\RR^2\minus\{p\}$ is
expressed via invariants $\ind(\Gamma)$, $\ind_p(\Gamma)$ of
degrees zero and one.
This fact is easy to explain. Let us show, that if a certain algebraic
number of lines satisfying some passage/tangency conditions is
a locally constant function $f$ on the space of generic immersed
curves, then it is an invariant of degree less than or equal to two.
Indeed, let $\Gamma_{sing}$ be an immersed curve with three
non-generic fragments contained in three small disks $\DD_i$, 
$i=1,2,3$, which do not lie on one line (i.e., no line passes 
through all three of them).
Suppose  that some line $l$ is counted for one of the resolutions
$\Gamma_\eps$ of $\Gamma_{sing}$. Then $l$ does not pass
through at least one of the disks, say, $\DD_1$. But then $l$ is
counted twice -- with opposite signs -- for both resolutions of
$\Gamma_{sing}$ inside $\DD_1$, hence its contribution to $f$
sums up to $0$, and we readily get $f(\Gamma_{sing})=0$.

By the same argument (noticing that no rational curves of degree
$d$ pass through $3d$ generic points), we immediately obtain the
following

\begin{thm}
Suppose that a certain algebraic number of real rational algebraic
plane curves of degree $d$, satisfying some passage/tangency
conditions, is a locally constant function on the space of generic
immersed curves. Then it is an invariant of degree less than or equal
to $3d-1$.
\end{thm}

Moreover, if a curve is required to pass through $k$ fixed points
(in general position), then an algebraic number of such curves is
an invariant of degree less than or equal to $3d-k-1$. In particular,
for $k=3d-2$ we get the upper bound one on the degree of an 
invariant. This explains the structure of formula
\eqref{eqn:main formula} of Theorem \ref{thm:main result}.


\end{document}